\documentclass{amsart}

\usepackage{amsmath}
\usepackage{amssymb}
\usepackage{amsthm}
\usepackage{url}
\usepackage{graphicx}
\newtheorem{theorem}{Theorem}
\newtheorem{lemma}[theorem]{Lemma}
\newtheorem{proposition}[theorem]{Proposition}
\newtheorem{corollary}[theorem]{Corrolary}

\theoremstyle{definition}
\newtheorem{definition}{Definition}
\newtheorem{remark}{Remark}
\newtheorem{notation}{Notation}
\newtheorem{example}{Example}

\DeclareMathOperator{\ent}{ent}
\newcommand{\eps}{\varepsilon}
\newcommand{\R}{\mathbb{R}}
\newcommand{\Canovas}{C{\'a}novas}
\newcommand{\Rodriguez}{Rodr{\'{\i}}guez}

\title{Topological entropy of compact subsystems of transitive real line maps}
\author{Dominik Kwietniak}

\address{Faculty of Mathematics and Computer Science, Jagiellonian University in Krak\'ow, ul. \L o\-jasiewicza 6, 30-348 Krak\'ow, Poland}\email{dominik.kwietniak@uj.edu.pl}

\author{Martha Ubik}
\address{Faculty of Mathematics and Computer Science, Jagiellonian University in Krak\'ow, ul. \L o\-jasiewicza 6, 30-348 Krak\'ow, Poland}\email{martha.ubik@gmail.com}

\date{\today}

\begin{document}

\begin{abstract}
For a continuous map $f$ from the real line (half-open interval $[0,1)$) into itself let $\ent(f)$ denote the supremum of topological entropies of $f|_K$, where $K$ runs over all compact $f$-invariant subsets of $\mathbb{R}$ ($[0,1)$, respectively). It is proved that if $f$ is topologically transitive, then the best lower bound of $\ent(f)$ is $\log\sqrt{3}$ ($\log 3$, respectively) and it is not attained.
This solves a problem posed by \Canovas\ [Dyn. Syst. \textbf{24} (2009), no. 4, 473--483].
\end{abstract}

\maketitle






\section{Introduction}
A question of considerable interest (see \cite{ABLM,AKLS,ALM,ALMM,Baldwin,BS,Blokh1,Canovas,Dirbak,HK,HKO2,KM,Spitalsky,SuYe,Ye}) is: how various properties of a dynamical system affect its topological entropy? To make this question precise one first fixes a class of dynamical systems, and then searches for bounds for topological entropy of maps in that class. In most studies the attention is restricted to topologically transitive systems, as transitivity is regarded as the simplest sufficient condition for non-trivial global dynamics. Also, in most cases one concentrates on compact metric spaces, as the notion of topological entropy is best suited to this setting.

It turns out that in general, even for compact metric spaces, there is no connection between topological transitivity and topological entropy. A system with positive topological entropy need not be transitive, and a transitive system may have zero topological entropy. However, there are spaces such that every topologically transitive map on them have necessarily positive topological entropy. For instance, by \cite{Blokh1} on a compact interval $[0,1]$ every transitive map has topological entropy at least $\log\sqrt{2}$, and there is a transitive map with topological entropy matching this bound. For references to these and other results of this type, e.g. lists of known best lower bounds for the topological entropy of transitive maps on various spaces see \cite[page 341]{ALM} or \cite{AKLS,BS,HK,HKO2,SuYe,Ye}.

Recently, \Canovas\ and \Rodriguez\ in \cite{CR} introduced entropy-type invariant for a dynamical system defined on not necessarily compact space. For a continuous map $f\colon X\mapsto X$, where $X$ denotes a topological space, the invariant $\ent(f)$ from \cite{CR} is defined as the supremum of topological entropies of $f|_K$, where $K$ runs over all compact invariant (meaning $f(K)\subset K$) subsets of the real line, with the agreement that $\sup\emptyset=\infty$.
This definition coincides with the standard one when applied to the compact dynamical system. Then the question about the connection between the invariant $\ent$ and other properties of dynamical systems arises naturally. One of possible problems of this kind has been studied by \Canovas\ in \cite{Canovas}, where a lower bound is obtained for the (non-compact) topological entropy $\ent$ of transitive maps of the real line. Unfortunately, the proof presented in \cite{Canovas} is flawed and contains at least two errors (some relevant counterexamples are indicated below, see Section \ref{sec:remarks}). Our goal is to present a correct proof of the existence of the lower bound for \Canovas-\Rodriguez\ entropy of transitive real line map, and solve the problem completely, showing that the bound from \cite{Canovas} is the best possible, that is, $\inf\{\ent(f):f\in \mathcal{T}(\mathbb{R})\}=\log\sqrt{3}$, where $\mathcal{T}(\mathbb{R})$ denotes the family of all transitive maps of the real line.
We also prove that the bound is not attained, that is, there is no transitive map $f$ of the real line with $\ent(f)=\log\sqrt{3}$.
These results are covered in Theorem \ref{thm:inf-is-log-sqrt-3}. This confirms the conjecture  of \Canovas\ from \cite{Canovas}.

We consider two similar problems: we prove in Theorem \ref{thm:inf-not-bi-is-log-sqrt-3} that $\inf\{\ent(f):f\in \mathcal{T}'(\mathbb{R})\}=\log\sqrt{3}$, where $\mathcal{T}'(\mathbb{R})$ denotes the family of all transitive, but not bitransitive maps of the real line and we solve an analogous problem for the maps from the half-open interval $[0,1)$ to itself showing in Theorem \ref{thm:inf-half-open} that $\inf\{\ent(f):f\in \mathcal{T}([0,1))\}=\log 3$.

The organization of the paper is the following. In Section \ref{sec:terminology} we set up notation and terminology.
The purpose of Section \ref{sec:aux} is to remind of some important properties of transitive maps from a real interval into itself and to quote some auxiliary results connecting the topological entropy with the notion of a horseshoe. In Section \ref{sec:lower-bounds} we prove the existence of lower bounds for $\ent(f)$ for $f$ in various classes of transitive maps, while in Section \ref{sec:examples} we construct examples proving that these bounds are the best possible. Section \ref{sec:remarks} contains our remarks and corrections to \cite{Canovas}.
In Section \ref{sec:specification} we share some observations on the specification property in the non-compact setting. We prove that for non-compact spaces the specification property is no longer a conjugacy invariant, and there are mixing maps of the open and half-open interval without the specification property.

\section{Terminology and notation}\label{sec:terminology}

A \emph{dynamical system} is a pair $(X,f)$ where $X$ is a metric space and $f$ is a map of $X$. Here, \emph{a map of $X$} means always \emph{a continuous map from $X$ to itself}. As usual, when the domain is clear, we will write about properties of maps, having in mind properties of underlying dynamical systems. In this convention, we say that a map $f$ is \emph{transitive} if for every nonempty open subsets $U$ and $V$ of $X$, the intersection $f^n(U)\cap V$ is nonempty for some positive integer $n$; a map $f$ is \emph{totally transitive} if for every natural $k>0$ its $k$-th iterate $f^k=f\circ\ldots\circ f$ ($k$-times composition of $f$ with itself) is transitive; in particular, a map $f$ is \emph{bitransitive} if $f^2$ is transitive; a map $f$ is \emph{weakly mixing} if $f\times f$ is transitive; finally, a map $f$ is \emph{mixing} if for every nonempty open subsets $U$ and $V$ of $X$, there is a positive integer $N$ such that the intersection $f^n(U)\cap V$ is nonempty for all $n\ge N$. If the underlying space is a dense in itself Baire space, then $(X,f)$ is transitive if and only if there is a point $x\in X$ whose \emph{orbit} $\{f^n(x):n\ge 0\}$ is dense in $X$ (see \cite{KS}).

To be more precise, we should write about \emph{topologically transitive (mixing etc.)} systems, to distinguish these notions from their ergodic counterparts, but we hope that no misunderstanding will arise, if we skip here the adverb \emph{topologically} to shorten the exposition. A set $K\subset X$ is invariant for $f$ if $f(K)\subset K$. Restricting a map $f$ to a non-empty invariant set $K$ we obtain a \emph{subsystem} of a system $(X,f)$. A dynamical system $(X,f)$ is a \emph{factor} of a system $(Y,g)$ if there is a continuous surjection $\varphi\colon Y\mapsto X$ such that $\varphi\circ g=f\circ\varphi$. In this case, we call the system $(Y,g)$ an \emph{extension} of $(X,f)$. If $\varphi$ as above is also a homeomorphisms, then we say that systems $(X,f)$ and $(Y,g)$ are \emph{conjugated}. 
For a dynamical system $(X,f)$ defined on a compact metric space
one may define the \emph{topological entropy} of the system, denoted $h(f)$. This nonnegative number from the extended interval $[0,+\infty]=[0,+\infty)\cup\{+\infty\}$ is an important conjugacy invariant. As we not need to appeal to the definition of topological entropy, we refer the reader to the literature (see \cite[Chapter 4]{ALM}). Let us only recall that for every compact dynamical system $(X,f)$ we have: $h(f^k)=kh(f)$ for each $k>0$, the entropy of any closed subsystem is not greater than $h(f)$; the  entropy of any factor do not exceed the entropy of the extension; finally, if $f$ is Lipschitz with constant $L\ge 0$, then $h(f)\le \max\{0,\log L\}$. If $X$ is not necessarily compact space, we follow \Canovas\ and \Rodriguez\ \cite{CR}, and define
\[
\ent (f)=\sup\{h(f|_K):K\in\mathcal{K}(f)\},
\]
where $\mathcal{K}(f)$ denotes the family of all non-empty compact $f$-invariant subsets of $X$.
Main properties of this entropy are stated in \cite[Theorem 2.1]{CR}. We note here only point~(c) of that Theorem:
$\ent(f^n)=n\cdot \ent(f)$ for all $n\ge 1$.
We may adopt a convention that $\ent(f)=\infty$ if $f$ has not any non-empty compact invariant subsets,
but we will not need it in the present paper anyway.

\section{Auxiliary results}\label{sec:aux}

\subsection{Quasihorseshoes and entropy}


For the proof of existence and non-attainability of the lower bound we need tools developed in \cite{HK}.

By a \emph{real interval} (an \emph{interval} for short) we mean a connected subset of the real
line with non-empty interior.
Any real interval $L$ is thus homeomorphic to one of the following subsets of the real line with
the usual topology: a \emph{compact interval} $[0,1]$, a \emph{half-open interval} $[0,1)$, or an
\emph{open interval} $(0,1)$. It follows that any dynamical system
on a real interval is conjugate to a system on one of the following spaces:
$[0,1]$, $[0,1)$, or $(0,1)$.

Let $f$ be a map from a real interval $L$ to $\mathbb{R}$. An
$s$-\emph{quasihorseshoe} for $f$ is a compact interval $J\subset
L$, and a collection $\mathcal{C}=\{A_1,\ldots,A_s\}$ of $s\ge 2$
nonempty compact subsets of $J$ fulfilling the following three
conditions: (a)~each set $A\in \mathcal{C}$ is an union of finite
number of compact intervals, (b)~the interiors of the sets from
$\mathcal{C}$ are pairwise disjoint, (c)~$J\subset f(A)$ for every
$A\in\mathcal{C}$. A quasihorseshoe $(J,\mathcal{C})$ is
\emph{tight} if $J$ is the union of elements of $\mathcal{C}$ and
$f(A)=J$ for every $A\in\mathcal{C}$. A quasihorseshoe
$(J,\mathcal{C})$ is \emph{loose} if the union of elements of
$\mathcal{C}$ is a proper subset of $J$. An $s$-quasihorseshoe (a
tight $s$-quasihorshoe) $(J,\mathcal{C})$ for $f$ is called an
$s$-\emph{horseshoe} (a tight $s$-\emph{horseshoe}) if every
$A\in\mathcal{C}$ is a compact interval. Our definition of a
horseshoe is equivalent to the definition from \cite[page 204]{ALM}.
If there is an $s$-quasihorseshoe ($s$-horseshoe, etc.)
$(J,\mathcal{C})$ for $f$, we simply say that $f$ \emph{has} an
$s$-quasihorseshoe ($s$-horseshoe, etc.), and $J$ \emph{carry} a
quasihorseshoe ($s$-horseshoe, etc.) for $f$.

It is straightforward to see that the proofs of \cite[Lemma
4.3.1]{ALM} and \cite[Proposition 4.3.2]{ALM} are, with the
necessary changes, valid for our quasihorseshoes instead of
horseshoes (see also \cite[Remark 4.3.4]{ALM}). For completeness we
reformulate Lemma 4.3.1 and Proposition 4.3.2 of \cite{ALM} and Proposition 4.8 of \cite{HK}
as Proposition~\ref{prop:horseshoe-gives-entropy} below.

\begin{proposition} \label{prop:horseshoe-gives-entropy}
If a transitive map $f$ of a real interval $L$ has a loose
$s$-quasihorseshoe then there exists $N>0$ such that for every $n\ge
N$ the map $f^n$ has an $(s^n+1)$-quasihorseshoe.
Additionally, there exists a compact invariant subset $K$ such that
$\ent(f)\ge h(f|_K)>\log s$.
\end{proposition}

We quote for future reference another result and its immediate consequence.

\begin{lemma}\label{lem:not-full}
If $(J,\mathcal{C})$ is an $s$-quasihorseshoe for a transitive map
$f$ of a real interval $L$ and $J\neq L$, then $f$ has a loose
$s$-quasihorseshoe $(J,\mathcal{D})$.
\end{lemma}
\begin{proof}
If $(J,\mathcal{C})$ is a tight $s$-quasihorseshoe, then $f(J)=J$, but
$J\neq L$ contradicts the transitivity.
\end{proof}

\begin{corollary}\label{cor:half-open-loose}
Let $f$ be a transitive map of open or half-open interval. Then every
quasihorseshoe for $f$ is loose.
\end{corollary}

\subsection{Properties of transitive maps from a real interval into itself}

We recall two propositions which generalize results given for
interval maps in \cite{BM-snake} and \cite{BM-do} (see also \cite[pp. 156--59]{BC}). %
They may be proved in much the same way as in original
references, or else can be deduced from previously known results as
noted in \cite[Section 7]{Banks}.
For other properties of transitive map of the real line see also \cite{NSS,NKSS}.

\begin{proposition}\label{prop:tot-trans-alternative}
Let $f$ be a transitive map of a real interval $J$. Then, exactly
one of the following statements holds:
\begin{enumerate}
 \item $f^2$ is transitive,
 \item there exist intervals $K,L\subset J$,
 with $K\cap L=\{c\}$ and $K\cup L=J$, such that $c$ is the unique fixed
 point for $f$,~$f(K)=L$ and $f(L)=K$.
\end{enumerate}
\end{proposition}

\begin{proposition}\label{prop:mixing-equivalences}
For a map $f$ of a real interval $L$ the following statements are
equivalent:
\begin{enumerate}
 \item $f$ is bitransitive, that is, $f^2$ is transitive,
 \item $f$ is totally transitive,
 \item $f$ is weakly mixing,
 \item $f$ is mixing,
 \item for every interval $J\subset L$, and for any compact
 interval $K$ contained in the interior of $L$ with respect to
 the natural topology of the real line there is an $N>0$ such
 that $K\subset f^n(J)$ for $n\ge N$.
\end{enumerate}
\end{proposition}

As an immediate consequence we obtain the following.

\begin{corollary}\label{cor:half-open-mix}
If $f$  is a transitive map of a half-open interval, then $f$ is mixing.
\end{corollary}

\section{Lower bounds}\label{sec:lower-bounds}

First we prove existence of the lower bound of \Canovas-\Rodriguez\ entropy for transitive maps of a half-open interval.

\begin{proposition}\label{prop:half-open-lower-bound}
If a map $g$ from the half-open interval $[0,\infty)$ to itself is transitive,
then $g$ has a loose $3$-horseshoe, hence
$\ent(g)> \log3$.
\end{proposition}
\begin{proof}
First note that if for some $x\in [0,\infty)$ we have $g(y)\ge y$ for all $y\in (x,\infty)$, then the interval $[y,\infty)$ is invariant for $g$ for every $y\in (x,\infty)$. This is not possible, since $g$ is transitive on $[0,\infty)$, so for every $x\in [0,\infty)$ there exists a point $y'\in [x,\infty)$ such that $g(y')<y'$.
Note also that for every point $x>0$ there exists a point $y''$ in $[0,x]$ such that $y''\le x < g(y'')$, as otherwise $[0,x]$ would be invariant for $g$, which contradicts the transitivity of $g$.
It follows that at least one point $z_1\in(0,\infty)$ is fixed for $g$.

We claim that in fact there must be an unbounded and increasing sequence $\{z_n\}$ of fixed points for $g$. To see this assume on contrary that $\bar{z}=\max\{x\in [0,\infty):g(x)=x\}$ exists. Clearly, $\bar{z}\ge z_1>0$. Then, either $g(x)>x$ for all $x>\bar{z}$, and as a consequence $[\bar{z},\infty)$ would be invariant for $g$, or $g(x)<x$ for all $x>\bar{z}$, and if we set $\omega=\max g([0,\bar{z}])$ then we would get $g$-invariant set $[0,\omega]$. In any case, we would arrive at contradiction with transitivity of $g$. This proves the claim.

Let $z_1>0$ be a fixed point of $g$ and define $a:=\max g([0, z_1])$. Clearly, $a>z_1$, as otherwise $[0,z_1]$ would be invariant for $g$. Consider $b:=\max g([0,a])$. By transitivity of $g$ on $[0,\infty)$ we see that $[0,a]$ can not be invariant for $g$, so $b>a$, moreover $b=\max g([z_1,a])$, since $b$ can not be attained in $[0,z_1]$. Let $p\in [z_1,a]$ be a point such that $g(p)=b$. We have $g(p)=b>a\geq p$, and there exists a maximal open interval $(u,v)$ containing $p$ such that for every $x$ in $(u,v)$ we have $g(x)>x$, and $u$ and $v$ are fixed points for $g$. Obviously, $u\in [z_1,a]$, and it is possible that $z_1=u$. There exists $y>v$ such that $g(y)<u$, as the interval $[u, +\infty)$ is not invariant for $g$ and $g(x)>u$ for $x\in (u,v]$. Choose $q$ to be the smallest element of the nonempty set $\{x\in [v,\infty):g(x)=u\}$. As $g(q)=u<q$,  we can find a maximal open interval $(w,y)$ containing $q$ such that $g(x)<x$ for every $x$ in $(w,y)$, and $w$ and $y$ are fixed for $g$. Let $d:=\max g([0,y])$ ($b=d$ is possible). By transitivity, $d>y$, and our previous definitions assert that $d=\max g([u,y])$, and there is a point $r\in [u,w]$ such that $g(r)=d$ (the latter comes from the fact that $g(x)\leq x\leq t$ for $x\in [w,y]$). Let $I=[u,r]$, $J=[r,q]$, $K=[q,y]$. It is clear that $[u,y]\subset g(I)\cap g(J) \cap g(K)$,
so $g$ has a 3-horseshoe, which is loose by Corollary \ref{cor:half-open-loose} and using Proposition \ref{prop:horseshoe-gives-entropy} we get $\ent(g)>\log3$, and the proof is finished.

\end{proof}
To show the existence of the lower bound for transitive maps of the real line we consider cases, depending on the number of fixed points of a map.

\begin{proposition}\label{prop:two-fixed}
If a transitive map of the real line $f$ has at least two fixed points,
then $f$ has a loose $2$-horseshoe, hence
$\ent(f)> \log{2}$.
\end{proposition}
\begin{proof}By Corollary \ref{cor:half-open-loose} and Proposition \ref{prop:horseshoe-gives-entropy} it is enough to find a $2$-horseshoe for $f$. To this end, note that the set of fixed points for a transitive map is always closed and nowhere dense. Therefore we can find fixed points $a$ and $b$ for $f$ such that $a<b$ and no fixed point of $f$ belongs to $(a,b)$. This implies that $f(x)>x$ for all $x\in(a,b)$ or $f(x)<x$ for all $x\in(a,b)$.  We assume that the former inequality holds for all $x$ in $(a,b)$. The later case can be handled the same way. By transitivity $[a,\infty)$ cannot be invariant for $f$, therefore there is a point $c>b$ such that $c=\min\{x>b:f(x)=a\}$.
Let $z=\max f([a,c])$, and let $d\in [a,c]$ be a point such that $f(d)=z$. We see that $z>c$, otherwise $[a,c]$ would be invariant for $f$. Now, $I=[a,d]$, and $J=[d,c]$ form a $2$-horseshoe for $f$ as claimed.
\end{proof}

\begin{proposition}\label{prop:unique-fix}
If a transitive map of the real line $f$ has a unique fixed point,
then $f^2$ has a loose $3$-horseshoe, hence
$\ent(f)> \log\sqrt{3}$.
\end{proposition}
\begin{proof}By Corollary \ref{cor:half-open-loose} and Proposition \ref{prop:horseshoe-gives-entropy} it is enough to find a $3$-horseshoe for $f^2$.
Let $z$ be the unique fixed point of $f$. It follows from transitivity and uniqueness of $z$ that
$f(x)>x$ for all $x<z$, and $f(x)<x$ for all $x>z$. Hence, $f([\alpha,z])\subset(\alpha,\infty)$ for any $\alpha<z$, and $f([z,\beta])\subset(-\infty,\beta)$ for any $z<\beta$. Let $A=(-\infty,z]$ and $B=[z,\infty)$. Clearly, $B\subset f(A)$ and $A\subset f(B)$. We have two cases:

\emph{Case I.} It holds $f(A)=B$, and $f(B)=A$, equivalently $f^2$ is not transitive. Applying Proposition \ref{prop:tot-trans-alternative} we can assert that $f^2$ restricted to $X=B$ is a transitive self-map of the half-open interval $B$. Let $g$ denote the map $f^2|_B$. Then Proposition \ref{prop:half-open-lower-bound} applies to $g$, hence
$g=f^2$ has a loose 3-horseshoe and we get $
\ent(f)=(1/2)\ent(f^2)>\log \sqrt{3}$, and the proof for the first case is finished.

\emph{Case II.} We have $B\subset f(A)$ and $A\subset f(B)$, but $B\subsetneq f(A)$, or $A\subsetneq f(B)$, equivalently $f^2$ is transitive. It follows that $z=f(a)$ for some $a\neq z$. Without loss of generality we assume $a<z$, that is, $B\subsetneq f(A)$. Let $b=\min\{x>z:f(x)=a\}$, so $f(x)\ge a$ for $x\in [a,b]$, in particular, $f([z,b])\subset [a,b]$. Hence, $c:=\max f([a,b])=\max f([a,z])$. If $c \le b$, then $f([a,b])\subset [a,b]$, violating transitivity. Therefore there is a point $p\in [a,z]$ such that $f(p)=c>b$. Moreover, $\max f([a,c])=c$, as $f(x)< x$ for $x\in [b,c]$. But $[a,c]$ can not be invariant for $f$, hence $\min f([a,c])<a$. It is clear that $d=\min f([a,c])$ must be attained at some point $q\in [b,c]$. Let $e=\max f([d,c])$. As above, we can see that $e=f(r)>c$ for some $r\in [d,a]$. We get that $f^2(b)=z$, and $[a,z]\subset f([z,b])$. Therefore there exists a point $s\in (z,b)$ such that $f(s)=p$, hence $f^2(s)=c$. Similarly, $[d,a]\subset f([b,c])$ and there exists a point $t\in (b,c)$ such that $f(t)=r$, hence $f^2(t)=e>c$. Let $I=[z,s]$, $J=[s,b]$, and $K=[b,t]$. Then the triple $(I,J,K)$ forms a loose $3$-horseshoe for $f^2$, and we get $\ent(f)=(1/2)\ent(f^2)>\log \sqrt{3}$, and the proof for the second case is finished.
\end{proof}

\begin{theorem}\label{thm:main-bound}
If $f$ is a transitive map of the real line, then $\ent(f) > \log \sqrt{3}$.
\end{theorem}
\begin{proof}Either $f$ has a unique fixed point, or there are at least two fixed points for $f$.
In the first case we invoke Proposition \ref{prop:unique-fix}, and in the second case we use Proposition \ref{prop:two-fixed} and observe that $\log 2 > \log \sqrt{3}$.
\end{proof}

\section{Examples}\label{sec:examples}

In this section we define examples showing that the bounds obtained in Proposition \ref{prop:half-open-lower-bound} and Theorem \ref{thm:main-bound} are best possible. These examples are also used as counterexamples to some claims from \cite{Canovas} (see Section \ref{sec:remarks}) and in Section \ref{sec:specification}.
\begin{notation}For the rest of this section we fix $\eps>0$ and choose any $\lambda>3$
such that $\log\lambda<\log 3+\varepsilon$. We define points 
\begin{align*}
p_1 &= \frac {1}{\lambda},          & q_1&=\frac{\lambda+1}{4\lambda},  &   p_2 &= \frac{\lambda-1}{2\lambda},&
p_3 &= \frac{\lambda+1}{2\lambda},  & q_2&=\frac{3\lambda-1}{4\lambda}, &   p_4 &= \frac{\lambda-1}{\lambda},
\end{align*}
and intervals
\begin{align*}
P_1 &= [ 0 ,p_1],& Q_1 &= [p_1,q_1],& Q_2 &= [q_1,p_2], & P_2 &= [p_2,p_3],& Q_3 &= [p_3,q_2],\\
Q_4 &= [q_2,p_4],& P_3 &= [p_4, 1 ],& R_1 &= [ 0 ,q_1], & R_2 &= [q_1,q_2],& R_3 &= [q_2, 1 ].
\end{align*}
\end{notation}

\begin{figure}
  \hfill
  \begin{minipage}[t]{.45\textwidth}
    \begin{center}
     \includegraphics[angle=0,width=0.95\textwidth]{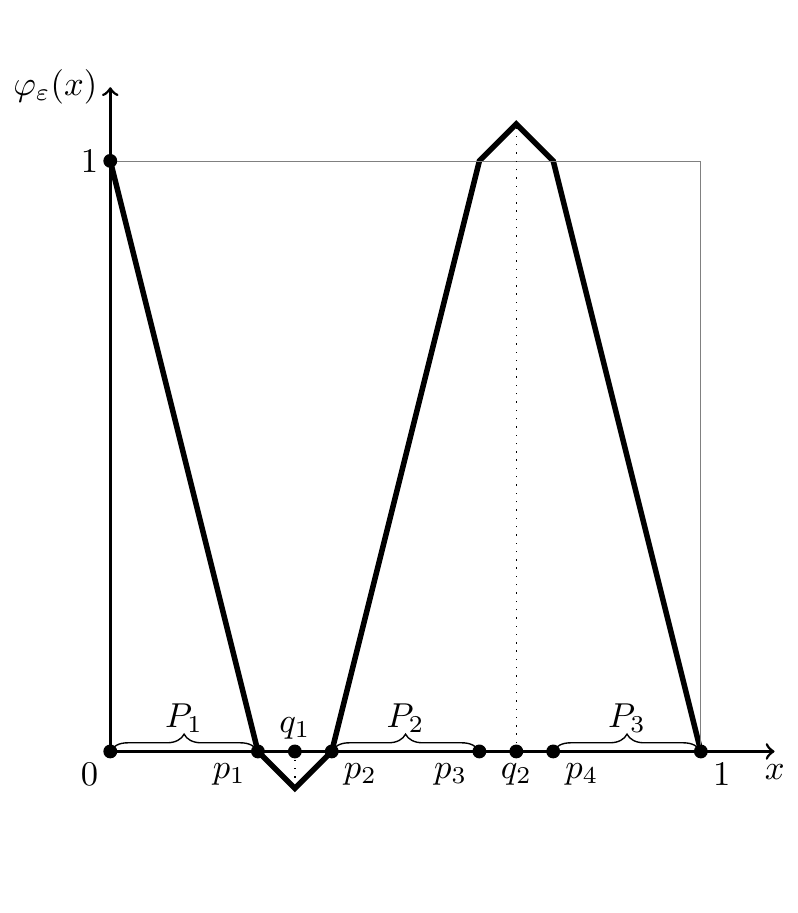}
      \caption{\label{fig:phi} Plot of $\varphi_\eps$.}
    \end{center}
  \end{minipage}
  \hfill
  \begin{minipage}[t]{.45\textwidth}
    \begin{center}
      \includegraphics[angle=0,width=0.95\textwidth]{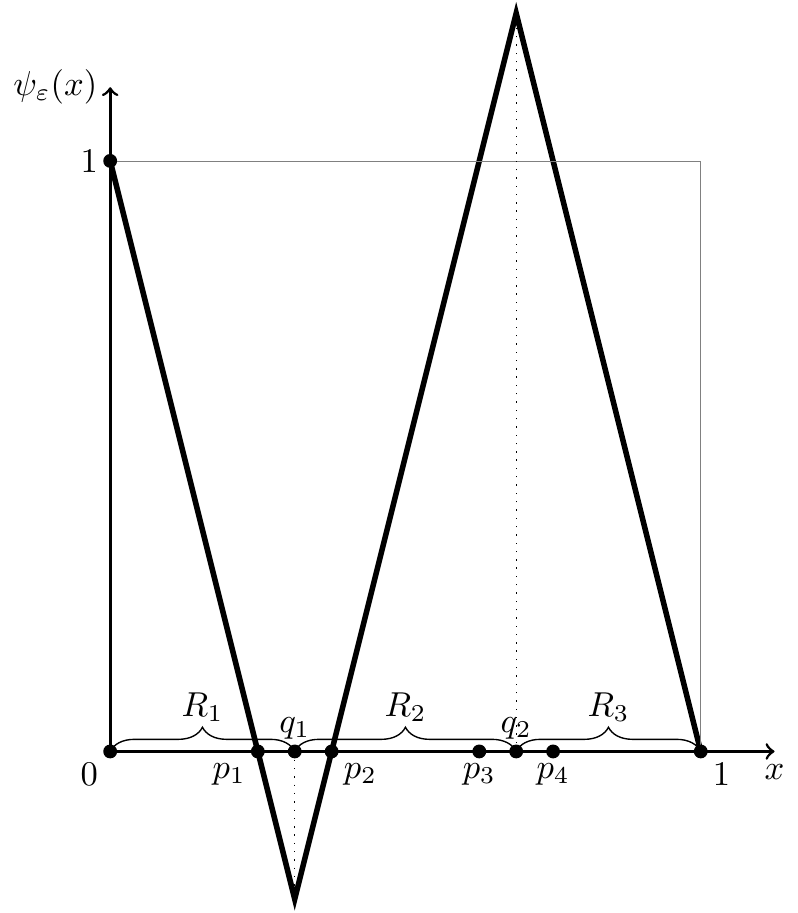}
      \caption{\label{fig:psi} Plot of $\psi_\eps$.}
    \end{center}
  \end{minipage}
  \hfill
\end{figure}

\begin{definition} Let $\varphi_\eps\colon [0,1]\mapsto \mathbb{R}$ be a map given by by the following formula (see Figure~\ref{fig:phi}):
\[
\varphi_\eps(x)=
\begin{cases}
1-\lambda x,&\text{for } x\in P_1,\\
\frac{1}{\lambda} - x,&\text{for } x\in Q_1,\\
x-\frac{\lambda-1}{2\lambda},&\text{for } x\in Q_2,\\
\lambda x-\frac{\lambda-1}{2},&\text{for }x\in P_2,\\
x+\frac{\lambda-1}{2\lambda},&\text{for }x\in Q_3,\\
\frac{2\lambda-1}{\lambda}-x,&\text{for }x\in Q_4,\\
\lambda-\lambda x,&\text{for }x\in P_3.
\end{cases}
\]
\end{definition}
\begin{definition}
Let $\psi_\eps\colon [0,1]\mapsto \mathbb{R}$ be a map given by the following formula (see Figure~\ref{fig:psi}):
\[
\psi_\eps(x)=
\begin{cases}
1-\lambda x,&\text{for } x\in R_1,\\
\lambda x-\frac{\lambda-1}{2},&\text{for } x\in R_2,\\ 
\lambda-\lambda x,&\text{for } x\in R_3.
\end{cases}
\]
\end{definition}

\begin{example}\label{ex:inf-is-log-sqrt-3}
We define a map $F_\eps\colon
\mathbb{R}\mapsto \mathbb{R}$ by setting
\begin{equation}\label{eq:extend}
F_\eps(x)=
\begin{cases}
-x,&\text{if }x>0,\\
\varphi(x-\lfloor x\rfloor) - \lfloor x\rfloor-1&\text{if }x\le 0,
\end{cases}
\end{equation}
where $\lfloor x\rfloor$ denotes the greatest integer function and gives the largest integer less than or equal to $x$
(see Figure \ref{fig:F_eps}).
\end{example}

\begin{figure}
  \hfill
  \begin{minipage}[t]{.45\textwidth}
    \begin{center}
     \includegraphics[angle=0,width=0.95\textwidth]{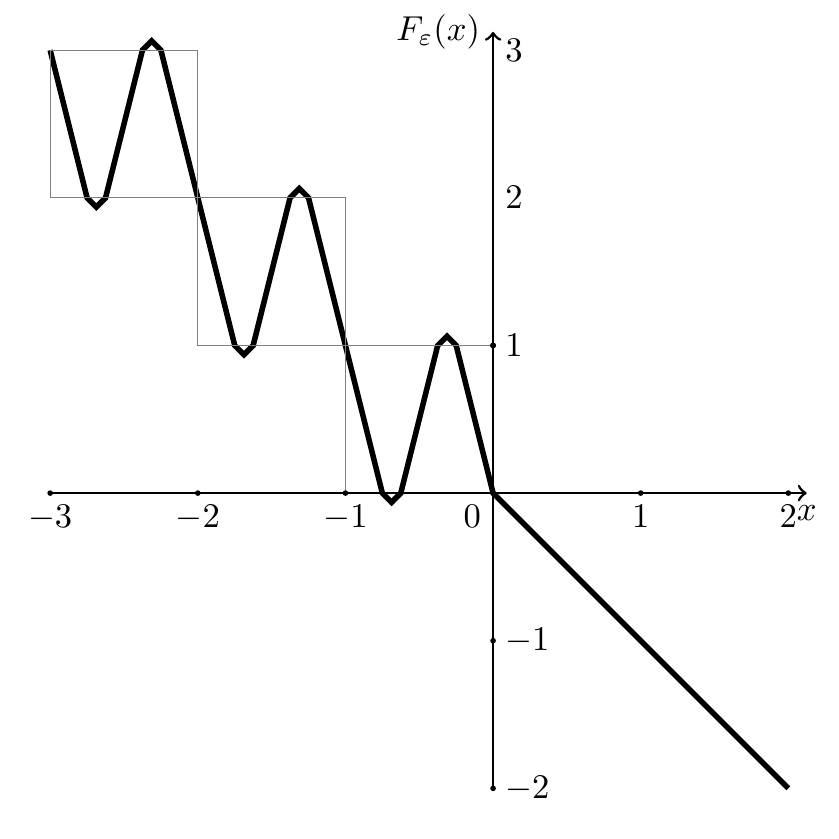}
      \caption{\label{fig:F_eps} Plot of $F_\eps$.}
    \end{center}
  \end{minipage}
  \hfill
  \begin{minipage}[t]{.45\textwidth}
    \begin{center}
      \includegraphics[angle=0,width=0.95\textwidth]{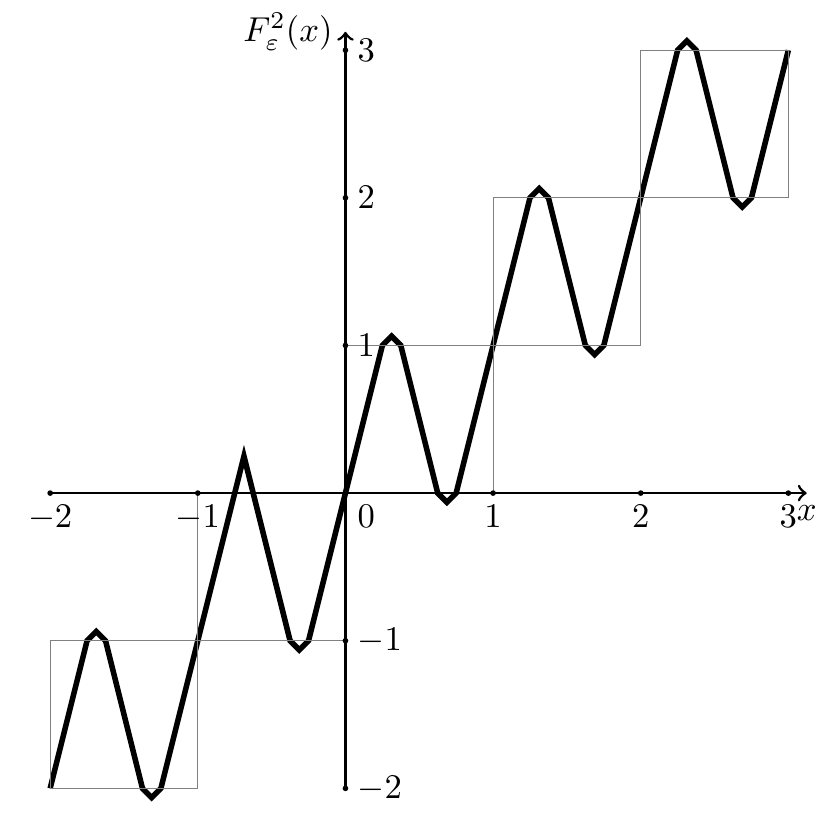}
      \caption{\label{fig:F2_eps} Plot of $F^2_\eps$.}
    \end{center}
  \end{minipage}
  \hfill
\end{figure}

\begin{proposition}\label{prop:best-bound-bitrans}
For every $\eps>0$ the map $F_\eps$ defined in Example \ref{ex:inf-is-log-sqrt-3} is mixing and its \Canovas-\Rodriguez\ entropy fulfils $\log\sqrt{3}< \ent(F_\eps)\le \log\sqrt{3}+\eps$.
\end{proposition}
\begin{proof}
Observe that $\varphi_\eps$ has slope $\pm\lambda$ on intervals $P_1$,
$P_2$, $P_3$, and slope $\pm 1$ on other intervals of
monotonicity. Moreover, it is easy to see that $F_\eps$, and $F^2_\eps$ are Lipschitz with
constant $\lambda$ (see Figures \ref{fig:F_eps} and \ref{fig:F2_eps}). Therefore for every compact $F^2_\eps$-invariant set $K\subset\mathbb{R}$
we have $h(F^2_\eps|_K)\le \log \lambda <\log 3 + \eps$. It follows that
for every compact $F_\eps$-invariant set $K\subset\mathbb{R}$ we have
$h(F_\eps|_K)=(1/2)h(F^2_\eps|_K)\le(1/2) \log \lambda \le (1/2)\log 3 + \eps$. Hence
$(1/2)\log 3 < \ent(F_\eps) \le (1/2)\log 3 + \eps$,
where the lower bound comes from Theorem \ref{thm:main-bound}.
It remains to observe that $F_\eps$ is mixing, since it is easy to see that
\begin{equation}\label{eqn:limes2}
\lim_{m\to\infty} (\inf F^{m}(K))=-\infty\textrm{ and
}\lim_{m\to\infty}(\sup F^{m}(K))=+\infty,
\end{equation}
for every compact interval $K\subset\mathbb{R}$.
\end{proof}

\begin{theorem}\label{thm:inf-is-log-sqrt-3}
Let $\mathcal{T}(\mathbb{R})$ denote the family of all transitive and continuous maps of the real line.
Then $\inf\{\ent(f):f\in \mathcal{T}(\mathbb{R})\}=\log\sqrt{3}$, and no map in $\mathcal{T}(\mathbb{R})$.
can attain this bound.
\end{theorem}
\begin{proof}It follows from Theorem \ref{thm:main-bound} that $\log\sqrt{3}$ is a lower bound, which can not be attained by any map from $\mathcal{T}(\mathbb{R})$. Proposition \ref{prop:best-bound-bitrans} shows that this bound is the best possible.
\end{proof}

\begin{example}\label{ex:inf-not-bi-is-log-sqrt-3}
First, divide the interval $(-\infty,0]$ into intervals
$\{I_k\}_{k=-\infty}^{+\infty}$ (overlapping at the endpoints),
where $I_k= \big[ -2^{-k},-2^{-k-1}  \big]$. Let $h_k$ denote an affine, orientation preserving homeomorphism,
which maps $[0,1]$ onto $I_k$. We define
\[
G_\eps(x)=\begin{cases}
h_{-k} \circ \psi \circ h^{-1}_k,&\text{if }x\in I_k\text{ for some }k\in\mathbb{Z},\\
-x,&\text{if }x>0.
\end{cases}
\]
See Figure \ref{fig:G_eps}.
\end{example}

\begin{figure}
  \hfill
  \begin{minipage}[t]{.45\textwidth}
    \begin{center}
     \includegraphics[angle=0,width=0.95\textwidth]{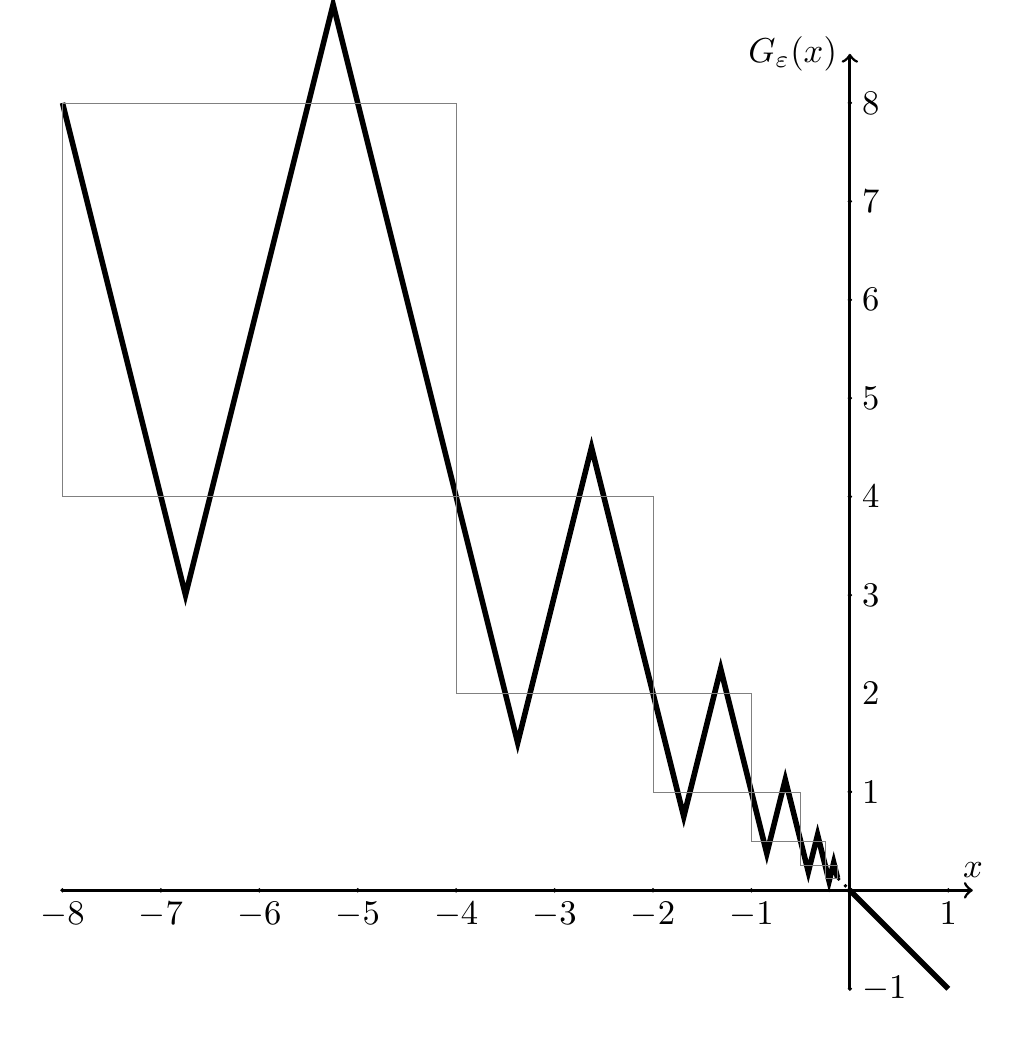}
      \caption{\label{fig:G_eps} Plot of $G_\eps$.}
    \end{center}
  \end{minipage}
  \hfill
  \begin{minipage}[t]{.45\textwidth}
    \begin{center}
      \includegraphics[angle=0,width=0.95\textwidth]{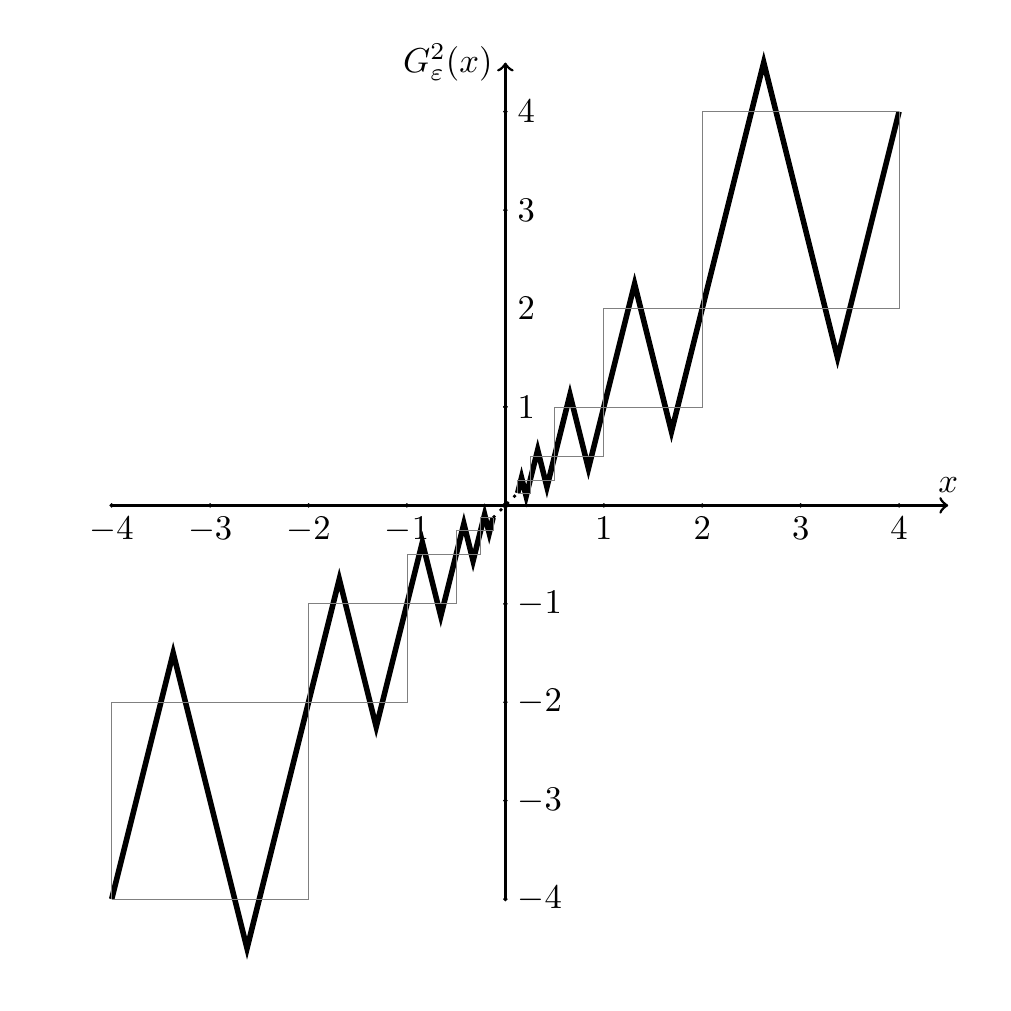}
      \caption{\label{fig:G2_eps} Plot of $G^2_\eps$.}
    \end{center}
  \end{minipage}
  \hfill
\end{figure}

\begin{proposition}\label{prop:best-bound-not-bi}
For every $\eps>0$ the map $G_\eps$ defined in Example \ref{ex:inf-not-bi-is-log-sqrt-3} is transitive, but not bitransitive, has a unique fixed point, and its \Canovas-\Rodriguez\ entropy fulfils $\log\sqrt{3}< \ent(G_\eps)\le \log\sqrt{3}+\eps$.
\end{proposition}
\begin{proof}
Observe that $\psi_\eps$ has slope $\pm\lambda$ on every interval of monotonicity,
hence $G_\eps$, and $G^2_\eps$ are Lipschitz with constant $\lambda$
(see Figures \ref{fig:G_eps} and \ref{fig:G2_eps}).
Therefore for every non-empty compact $G^2_\eps$-invariant set $K\subset\mathbb{R}$
we have $h(G^2_\eps|_K)\le \log \lambda <\log 3 + \eps$. It follows that
for every compact $G_\eps$ invariant set $K\subset\mathbb{R}$ we have
$h(G_\eps|_K)=(1/2)h(G^2_\eps|_K)\le(1/2) \log \lambda \le (1/2)\log 3 + \eps$. Hence
$(1/2)\log 3 < \ent(G_\eps) \le (1/2)\log 3 + \eps$,
where the lower bound comes from Proposition \ref{prop:unique-fix}.
It is also clear that $G_\eps([0,\infty))=(-\infty,0]$ and $G_\eps((-\infty,0])=[0,\infty)$, so $[0,\infty)$ and $(-\infty,0]$ are invariant subsystems for $G^2_\eps$. Moreover, $G^2_\eps|_{[0,\infty)}$ and $G^2_\eps|_{(-\infty,0]}$ are transitive. Hence $G_\eps$ is transitive, while $G^2_\eps$ is not.
\end{proof}

\begin{theorem}\label{thm:inf-not-bi-is-log-sqrt-3}
Let $\mathcal{T}'(\mathbb{R})$ denote the family of all transitive, but not bitransitive, and continuous maps of the real line.
Then $\inf\{\ent(f):f\in \mathcal{T}'(\mathbb{R})\}=\log\sqrt{3}$, and there is no map in $\mathcal{T}'(\mathbb{R})$ attaining this bound.
\end{theorem}
\begin{proof}
It follows from Theorem \ref{thm:main-bound} that $\log\sqrt{3}$ is a lower bound, which can not be attained by any map from $\mathcal{T}'(\mathbb{R})$. Proposition \ref{prop:best-bound-not-bi} shows that this bound is the best possible.
\end{proof}

\begin{theorem}\label{thm:inf-half-open}
Let $\mathcal{T}([0,\infty))$ denote the family of all transitive and continuous maps of the half-open interval.
Then $\inf\{\ent(f):f\in \mathcal{T}([0,\infty))\}=\log 3$, and there is no map in $\mathcal{T}([0,\infty))$ attaining this bound.
\end{theorem}
\begin{proof}
It follows from Proposition \ref{prop:half-open-lower-bound} that $\log 3$ is a lower bound, which can not be attained by any map from $\mathcal{T}([0,\infty))$. Let $G^2_\eps|_{[0,\infty)}$ be the restriction of the second iterate of the map defined in Example~\ref{ex:inf-not-bi-is-log-sqrt-3} to the half-open interval $[0,\infty)$. From the proof of Proposition \ref{prop:best-bound-not-bi}
we know that for every compact $G^2_\eps$-invariant set $K\subset [0,\infty)$ we have
$h(G^2_\eps|_K)\le \log \lambda \le \log 3 + \eps$, which shows that $\log 3$ is the best possible bound for \Canovas-\Rodriguez\ entropy of maps from $\mathcal{T}([0,\infty))$.
\end{proof}

\section{Remarks to \Canovas\ article \cite{Canovas}} \label{sec:remarks}

\begin{remark}
Part $(d)$ of Theorem 1 in \cite{Canovas} is false. It is in general not true that a bitransitive map of the interval $[0,1]$ has to have at least two fixed points, nor that such a map has topological entropy necessarily greater or equal to $\log 2$. The author of \cite{Canovas} compiled Theorem 1 from various results scattered through literature (he cites \cite{KS} and \cite[Chapter VII]{BC}) and mistakenly overstated it, as there is no such theorem neither in \cite{KS}, nor in \cite{BC}. Actually, the correct statement should be the following (for more details we refer the reader to Proposition 4.3.9 and Example 4.4.5 of \cite{Ruette}).
\begin{proposition} For every $\eps>0$ there exists a bitransitive map $f\colon [0,1]\mapsto [0,1]$ with a unique fixed point and topological entropy $h(f)\in (\log \sqrt{2}, \log \sqrt{2} +\eps)$.
\end{proposition}
The mistake described above leads to part $(d)$ of Theorem 5 in \cite{Canovas}, which says that the set of fixed points of bitransitive map of the real line must be unbounded. This conclusion is then used in proof of part $(b)$ of Theorem 4 in \cite{Canovas} stating for a bitransitive map $f\colon\R\mapsto\R$ we have $\ent(f)\ge \log2$. Both statements are false, as can be deduced from Example \ref{ex:inf-is-log-sqrt-3} and Theorem \ref{thm:inf-is-log-sqrt-3} (see also Figures \ref{fig:F_eps} and \ref{fig:F2_eps}).
\end{remark}

\begin{remark}
Although part $(a)$ of Theorem 4 on \cite{Canovas} is correct, its proof on \cite[page ]{Canovas} is not. The proof begins with the following claim: \emph{if $f$ is a transitive, but not bitransitive map of the the real line and $a$ denotes its unique fixed point, then there is a point $b>a$ such that $f^2(b)=a$}.
This is not true, as can be deduced from Example \ref{ex:inf-not-bi-is-log-sqrt-3} and Theorem \ref{thm:inf-not-bi-is-log-sqrt-3} (see also Figures \ref{fig:G_eps} and \ref{fig:G2_eps}). The correct proof of part $(a)$ of \cite[Theorem 4]{Canovas}
is presented above in the first case considered in Proposition \ref{prop:unique-fix}.
\end{remark}

\section{On specification property}\label{sec:specification}

The specification property was introduced by Bowen in
\cite{BowenSpec} (see also \cite{Denker}).
We say that $f\colon X\mapsto X$ has the \emph{periodic specification property} 
if, for any $\eps > 0$, there is an integer $N_\eps>0$
such that for any integer $s\geq 2$, any set $\{y_1,\dots,y_s\}$
of $s$ points of $X$, and any sequence $0=j_1\leq k_1 < j_2 \leq k_2
< \dots < j_s \leq k_s$ of $2s$ integers with $j_{l+1} - k_l\geq
N_\eps$ for $l= 1,\dots,s-1$, there is a point $x\in X$ such
that, for each $1\leq m\leq s$ and any $i$ with
$j_m \leq i \leq k_m$, the following conditions hold:
\begin{align}
d(f^i(x),f^i(y_m))&<\eps, \label{cond:psp1}\\
f^n (x)&=x, \quad\textrm{ where } \;n=N_\eps+
k_s. \label{cond:psp2}
\end{align}
If we drop the periodicity condition \eqref{cond:psp2} from the above definition,
that is, if $f$ fulfills only the first condition above, then we say that $f$ has the
\emph{specification property}.

\begin{remark}
Maps with the specification property are not necessarily surjective. To see this,
consider the discrete metric space $X=\{a,b\}$, and $f\colon X\mapsto X$ given by $f(a)=f(b)=a$.
It is not hard to verify that the map $f$ has the specification property.
\end{remark}

Maps with the periodic specification property have dense set of periodic points, hence such maps are onto,
and it easy to see that they are mixing. There are examples of mixing dynamical systems
with dense set of periodic points but without the specification property (see \cite{Weiss}). It was proved by Blokh
\cite{Blokh,Blokh2} (see \cite{Buzzi} for another proof), that mixing maps of a compact interval have the periodic specification property.
We note here that mixing map of the real line not necessarily have the specification property, and there are
conjugate dynamical systems defined on non-compact metric spaces such that one has the specification property, while the other has not, that is, specification property is not a conjugacy invariant outside the compact setting.

\begin{proposition}\label{prop:no-specification}
For every $\eps>0$ the map $F_\eps$ defined in Example \ref{ex:inf-is-log-sqrt-3} is a mixing map of the real line with the usual metric which has not the specification property, but there exists a map $f_\eps\colon(0,1)\mapsto(0,1)$ with the specification property, which is conjugate to $F_\eps$.
\end{proposition}
\begin{proof}We have already proved that $F_\eps$ is mixing. Note that if $x\in [-n,n]$ for some integer $n>0$, then $F_\eps(x)\in[-n-1,n+1]$. It follows that any point $x\in\mathbb{R}$ needs at least $n-1$ iterates of $F_\eps$ to travel from $1/2$ neighborhood of the orbit of fixed point $0$ to the $1/2$ neighborhood of the orbit of periodic point $n+1$, so $F_\eps$ can not have the specification property. It is easy to see that there is a mixing map of the compact interval $\bar{f}_\eps$ such that $f_\eps=\bar{f}_\eps|_{(0,1)}$ is conjugate to $F_{\eps}$. As every mixing map of the compact interval has the specification property, so does $\bar{f}_\eps$. It follows that $f_\eps$ also has this property for if $0$ or $1$ is needed to play the role of $z$ in definition of the specification property, then it can be replaced by sufficiently close periodic point of period $2$ lying in $(0,1)$. This finishes the proof.
\end{proof}

The following result may be proved much in the same way as Proposition \ref{prop:no-specification}.

\begin{proposition}
For every $\eps>0$ the map $H_\eps=G^2_\eps|_{[0,\infty)}$, where $G_\eps$ is the map defined in Example \ref{ex:inf-not-bi-is-log-sqrt-3} is a mixing map of the half-open interval $[0,\infty)$ with the usual metric which has not the specification property, but there exists a map $g_\eps\colon[0,1)\mapsto[0,1)$ with the specification property, which is conjugate to $H_\eps$.
\end{proposition}

We close this paper by offering a question for further research: \emph{Is the bound obtained in Proposition \ref{prop:two-fixed} best possible?}

\section*{Acknowledgement(s)}
The research leading to this paper were supported by the Polish Ministry of Science and Higher Education grant Iuventus Plus IP~$2011028771$.


\begin{thebibliography}{1}

\def\soft#1{\leavevmode\setbox0=\hbox{h}\dimen7=\ht0\advance \dimen7
  by-1ex\relax\if t#1\relax\rlap{\raise.6\dimen7
  \hbox{\kern.3ex\char'47}}#1\relax\else\if T#1\relax
  \rlap{\raise.5\dimen7\hbox{\kern1.3ex\char'47}}#1\relax \else\if
  d#1\relax\rlap{\raise.5\dimen7\hbox{\kern.9ex \char'47}}#1\relax\else\if
  D#1\relax\rlap{\raise.5\dimen7 \hbox{\kern1.4ex\char'47}}#1\relax\else\if
  l#1\relax \rlap{\raise.5\dimen7\hbox{\kern.4ex\char'47}}#1\relax \else\if
  L#1\relax\rlap{\raise.5\dimen7\hbox{\kern.7ex
  \char'47}}#1\relax\else\message{accent \string\soft \space #1 not
  defined!}#1\relax\fi\fi\fi\fi\fi\fi}
  \def\soft#1{\leavevmode\setbox0=\hbox{h}\dimen7=\ht0\advance \dimen7
  by-1ex\relax\if t#1\relax\rlap{\raise.6\dimen7
  \hbox{\kern.3ex\char'47}}#1\relax\else\if T#1\relax
  \rlap{\raise.5\dimen7\hbox{\kern1.3ex\char'47}}#1\relax \else\if
  d#1\relax\rlap{\raise.5\dimen7\hbox{\kern.9ex \char'47}}#1\relax\else\if
  D#1\relax\rlap{\raise.5\dimen7 \hbox{\kern1.4ex\char'47}}#1\relax\else\if
  l#1\relax \rlap{\raise.5\dimen7\hbox{\kern.4ex\char'47}}#1\relax \else\if
  L#1\relax\rlap{\raise.5\dimen7\hbox{\kern.7ex
  \char'47}}#1\relax\else\message{accent \string\soft \space #1 not
  defined!}#1\relax\fi\fi\fi\fi\fi\fi}
  \def\soft#1{\leavevmode\setbox0=\hbox{h}\dimen7=\ht0\advance \dimen7
  by-1ex\relax\if t#1\relax\rlap{\raise.6\dimen7
  \hbox{\kern.3ex\char'47}}#1\relax\else\if T#1\relax
  \rlap{\raise.5\dimen7\hbox{\kern1.3ex\char'47}}#1\relax \else\if
  d#1\relax\rlap{\raise.5\dimen7\hbox{\kern.9ex \char'47}}#1\relax\else\if
  D#1\relax\rlap{\raise.5\dimen7 \hbox{\kern1.4ex\char'47}}#1\relax\else\if
  l#1\relax \rlap{\raise.5\dimen7\hbox{\kern.4ex\char'47}}#1\relax \else\if
  L#1\relax\rlap{\raise.5\dimen7\hbox{\kern.7ex
  \char'47}}#1\relax\else\message{accent \string\soft \space #1 not
  defined!}#1\relax\fi\fi\fi\fi\fi\fi}
\providecommand{\bysame}{\leavevmode\hbox to3em{\hrulefill}\thinspace}
\providecommand{\MR}{\relax\ifhmode\unskip\space\fi MR }
\providecommand{\MRhref}[2]{%
  \href{http://www.ams.org/mathscinet-getitem?mr=#1}{#2}
}
\providecommand{\href}[2]{#2}
\bibitem{ABLM} Ll.~Alsed{\`a}, S.~Baldwin, J.~Llibre, M.~Misiurewicz,  \emph{Entropy of transitive tree maps}, Topology \textbf{36} (1997), no. 2, 519--532. \MR{1415604 (98f:54031)}

\bibitem{AKLS} Ll. Alsed{\`a}, S.~Kolyada, J.~Llibre, and {\soft{L}}.~Snoha, \emph{Entropy and
  periodic points for transitive maps}, Trans. Amer. Math. Soc. \textbf{351}
  (1999), no.~4, 1551--1573. \MR{1451592 (99f:58117)}

\bibitem{ALM} Ll. Alsed{\`a}, J. Llibre, and M. Misiurewicz,
  \emph{Combinatorial dynamics and entropy in dimension one}, second ed.,
  World Scientific, River Edge, NJ, 2000. \MR{1807264 (2001j:37073)}
\bibitem{ALMM} Ll.~Alsedà, J.~Llibre, F.~Ma{\~n}osas and M.~Misiurewicz, \emph{Lower bounds of the topological entropy for continuous maps of the circle of degree one}, Nonlinearity \textbf{1} (1988), no. 3, 463--479. \MR{0955624 (89m:58119)}
\bibitem{Baldwin} S.~Baldwin, \emph{Entropy estimates for transitive maps on trees}, Topology \textbf{40} (2001), no. 3, 551--569. \MR{1838995 (2002j:37025)}
\bibitem{Banks} J.~Banks, \emph{Regular periodic decompositions for topologically transitive maps}, Ergodic Theory Dynam. Systems \textbf{17} (1997), no. 3, 505--529. \MR{1452178 (98d:54074)}
\bibitem{BC} L.S.~Block and W.A.~Coppel, \emph{Dynamics in one dimension}, Lecture Notes in Mathematics, 1513. Springer-Verlag, Berlin, 1992. \MR{1176513 (93g:58091)}

 \bibitem{BS} F.~Balibrea and \soft{L}~Snoha, \emph{Topological entropy of Devaney chaotic maps}, Topology Appl. \textbf{133} (2003), no. 3, 225--239. \MR{2000500 (2004f:37018)}

 \bibitem{BM-snake} M.~Barge and J.~Martin, \emph{Chaos, periodicity, and snakelike continua},
  Trans. Amer. Math. Soc. \textbf{289} (1985), no.~1, 355--365. \MR{779069 (86h:58079)}


\bibitem{BM-do} 
\bysame, \emph{Dense orbits on the interval}, Michigan Math. J. \textbf{34}
  (1987), no.~1, 3--11. \MR{873014 (88c:58031)}


\bibitem{Blokh} A.~M.~Blokh,  \emph{Decomposition of dynamical systems on an interval}, (Russian) Uspekhi Mat. Nauk \textbf{38} (1983), no. 5(233), 179--180. \MR{0718829 (86d:54060)}. \emph{Erratum}, Uspekhi Mat. Nauk \textbf{42} (1983), no. 6(258), 233.
\bibitem{Blokh1} \bysame, \emph{On the connection between entropy and transitivity for one-dimensional mappings}, (Russian) Uspekhi Mat. Nauk \textbf{42} (1987), no. 5(257), 209--210. \MR{0928783 (89g:58117)}
\bibitem{Blokh2} \bysame, \emph{The ``spectral'' decomposition for one-dimensional maps}, Dynamics reported, 1--59, Dynam. Report. Expositions Dynam. Systems (N.S.), \textbf{4}, Springer, Berlin, 1995.
\bibitem{BowenSpec}
R.~Bowen, \emph{Topological entropy and axiom {A}}, in ''Global
Analysis'',
  Proceedings of Symposia on Pure Mathematics, vol.~14, Am. Math. Soc.,
  Providence, 1970.
\bibitem{Buzzi} J\'{e}r\^{o}me Buzzi  \emph{Specification on the interval}, Trans. Amer. Math. Soc. \textbf{349} (1997), no. 7, 2737--2754. \MR{1407484 (97i:58043)}
 \bibitem{Canovas} Jose S. \Canovas, \emph{Topological entropy of continuous transitive maps on the real line}, Dyn. Syst. \textbf{24} (2009), no. 4, 473--483. \MR{2572999 (2010j:37026)}
 \bibitem{CR} J.S. \Canovas, J.M. \Rodriguez, \emph{Topological entropy of maps on the real line}, Topology Appl. \textbf{153} (2005), no. 5-6, 735--746. \MR{2201485 (2006i:37086)}
\bibitem{Denker} M.~Denker, C.~Grillenberger, and K.~Sigmund, \emph{Ergodic theory on
compact
  spaces},  Lecture Notes in Mathematics, \textbf{527}, Springer-Verlag, Berlin-New York, 1976.

\bibitem{Dirbak} M. Dirb{\'a}k, \emph{Extensions of dynamical systems without increasing the entropy}, Nonlinearity \textbf{21} (2008), no. 11, 2693--2713. \MR{2448237 (2009h:37027)}
\bibitem{HK} G. Hara{\'n}czyk and D. Kwietniak, \emph{When lower entropy implies stronger Devaney chaos}, Proc. Amer. Math. Soc. \textbf{137} (2009), no. 6, 2063--2073. \MR{2480288}
\bibitem{HKO2} G. Hara{\'n}czyk, D. Kwietniak and P.~Oprocha \emph{Topological structure and entropy of mixing graph maps}, submitted, \texttt{arXiv:1111.0566v1 [math.DS]}
\bibitem{KS} S.~Kolyada and \soft{L}.~Snoha, \emph{Some aspects of topological transitivity—a survey}, Iteration theory (ECIT 94) (Opava), 3--35, Grazer Math. Ber., \textbf{334}, Karl-Franzens-Univ. Graz, Graz, 1997. \MR{1644768}
\bibitem{KM} D. Kwietniak and M. Misiurewicz, \emph{Exact {D}evaney chaos and
  entropy}, Qual. Theory Dyn. Syst. \textbf{6} (2005), no.~1, 169--179.
  \MR{2273492 (2007i:37031)}
\bibitem{NSS} A.~Nagar and S.P.~Sesha Sai, \emph{Some classes of transitive maps on $\mathbb{R}$}, J. Anal. \textbf{8} (2000), 103--111. \MR{1806399 (2002i:37051)}
\bibitem{NKSS} A.~Nagar, V. Kannan and S.P.~Sesha Sai, \emph{Properties of topologically transitive maps on the real line}, Real Anal. Exchange \textbf{27} (2001/02), no. 1, 325--334. \MR{1887863 (2002k:37014)}
\bibitem{Ruette} S. Ruette, \emph{Chaos for continuous interval maps --- a survey of relationship between the various sorts of chaos (preliminary version)}, \texttt{http://www.math.u-psud.fr/~ruette/articles/chaos-int.pdf} accessed 8.18.2012.
\bibitem{Spitalsky} V. {\v{S}}pitalsk{\'{y}}, \emph{Entropy and exact Devaney chaos on totally regular continua}, preprint, \texttt{	arXiv:1112.6017v2 [math.DS]}
\bibitem{SuYe} Y.~Su and X.~Ye, \emph{Topological entropy of a class of transitive maps of a tree}, Proceedings of the 14th Summer Conference on General Topology and its Applications (Brookville, NY, 1999). Topology Proc. \textbf{24} (1999), Summer, 597--609 (2001). \MR{1876390 (2002j:37026)}
\bibitem{Weiss} B. Weiss, \emph{Topological transitivity and ergodic measures}, Math. Systems Theory \textbf{5} (1971), 71--75. \MR{0296928 (45 \#5987)}
\bibitem{Ye} X.~Ye,  \emph{Topological entropy of transitive maps of a tree}, Ergodic Theory Dynam. Systems \textbf{20} (2000), no. 1, 289--314. \MR{1747021 (2002b:37052)}
\end{thebibliography}
\end{document}